\documentclass[11pt,reqno]{amsart}
\usepackage{graphicx}
\usepackage[pdftex]{hyperref}
\usepackage{subfigure}
\usepackage{amsthm}
\usepackage[margin=2cm]{geometry}
\usepackage{palatino}

\newtheorem{theorem}{Theorem}[section]
\newtheorem{lemma}[theorem]{Lemma}

\theoremstyle{definition}

\theoremstyle{remark}

\numberwithin{equation}{section}

\begin{document}

\title{An improved Popoviciu-type inequality for a new Bernstein-type operator}

\author{Mihai N. Pascu}
\address{Faculty of Mathematics and Computer Science, Transilvania University of Bra\c{s}ov, Str. Iuliu Maniu 50, Bra\c{s}ov -- 500091, Romania.}
\email{mihai.pascu@unitbv.ro}
\thanks{Supported by a grant of the Romanian National Authority for Scientific Research, CNCS - UEFISCDI, project number PNII-ID-PCCE-2011-2-0015.}

\author{Nicolae R. Pascu}
\address{Department of Mathematics, Kennesaw State University, 1100 S. Marietta Parkway, Marietta, GA 30060-2896, U.S.A.}
\email{npascu@kennesaw.edu}

\author{Floren\c{t}a Trip\c{s}a}
\address{Faculty of Mathematics and Computer Science, Transilvania University of Bra\c{s}ov, Str. Iuliu Maniu 50, Bra\c{s}ov -- 500091, Romania.}
\email{florentatripsa@yahoo.com}

\begin{abstract}
Recently we introduced a new Bernstein-type operator using P\'{o}lya's urn model with negative replacement, and we showed that it satisfies a Popoviciu-type inequality with a constant slightly larger than that of the corresponding inequality for the classical Bernstein operator.

In the present paper we prove an inequality for the rising factorial (of independent interest), and we use it in order to show that the constant in the Popoviciu inequality for the new operator is in fact smaller than the corresponding constant for the Bernstein operator.
\end{abstract}



\maketitle

\section{Introduction}
It is known that the classical Bernstein operator (\cite{Berstein 1912}) defined by
\begin{equation}\label{Probabilistic repr of Bernstein polynomial}
B_{n}\left( f;x\right) =\sum_{k=0}^{n}f\left( \frac{k}{n}\right)
C_{n}^{k}x^{k}\left( 1-x\right) ^{n-k} 
\end{equation}%
satisfies the inequality
\begin{equation}  \label{Popoviciu's error estimate for Bernstein}
\left\vert B_{n}\left( f;x\right) -f\left( x\right) \right\vert \leq C\omega
\left( n^{-1/2}\right) ,\qquad x\in \left[ 0,1\right] , \; n=1,2,\ldots,
\end{equation}%
where $f:[0,1]\rightarrow \mathbb R$ is an arbitrary continuous function, $\omega(\cdot)$ denotes the modulus of continuity of $f$.

T. Popoviciu (\cite{Popoviciu}) proved the above inequality for the value of the constant $C=\frac32$. Lorentz (\cite{Lorentz}, pp. 20 --21) improved the value of the constant to $C=\frac{5}{4}$, and also showed that the constant $C$ cannot be less than one. The optimal value of the constant $C$ for which the above inequality holds true for any continuous function was obtained by Sikkema (\cite{Sikkema}), who obtained the value
\begin{equation}
C_{opt}=\frac{4306+837\sqrt{6}}{5932}\approx 1.0898873...,
\label{Sikkema optimal constant}
\end{equation}%
attained in the case $n=6$ for a particular function.

Recently (\cite{PPT1}), we introduced the Bernstein-type operator $R_{n}$ defined by (\ref{Rational Bernstein operator}) and we showed that its also satisfies a Popoviciu-type inequality, with the constant $C=\frac{31}{27}\approx 1.14815$, smaller than Popoviciu's and Lorentz's constants, but slightly larger than Sikkema's optimal constant.

In the present paper we refine this result, by showing that the constant in the Popoviciu-type estimate is smaller than Sikkema's constant (Theorem \ref{thm for upper bound for optimal constant}). The proof is based on a certain inequality for the rising factorial (Lemma \ref{Monotonicity of F^c}) which is of independent interest.

\section{Preliminaries}\label{Preliminaries}

For $x,h\in \mathbb{R}$ and $n\in \mathbb{N}$ we set
\begin{equation}
x^{\left( n,h\right) }=x\left( x+h\right) \left( x+2h\right) \cdot \ldots
\cdot \left( x+\left( n-1\right) h\right)   \label{rising factorial}
\end{equation}%
for the generalized (rising) factorial with increment $h$. When $n=0$ we are using the convention $x^{\left( 0,h\right) }=1$ for any $x,h\in \mathbb{R}$.


A random variable $X_{n}^{a,b,c}$ has a P\'{o}lya's urn distribution (also known as P\'{o}lya-Eggenberger distribution, see \cite{Eggenberger-Polya}, \cite{Polya}) with parameters $ n\geq 1$, $a,b\in \mathbb{R}_{+}$, and $c\in\mathbb{R}$ satisfying \begin{equation}\label{Hypothesis on c}
a+\left( n-1\right) c\geq 0\qquad \text{and }\qquad b+\left( n-1\right)
c\geq 0,
\end{equation}
if it is given by (see for example \cite{Johnson-Kotz})
\begin{equation}
P\left(X_n^{a,b,c}=k\right)=p_{n,k}^{a,b,c}=C_{n}^{k}\frac{\left( a\right) ^{\left( k,c\right) }\left(
b\right) ^{\left( n-k,c\right) }}{\left( a+b\right) ^{\left( n,c\right) }}%
,\qquad k\in \left\{ 0,1,\ldots ,n\right\} .  \label{Polya urn probabilities}
\end{equation}

In the case when $a,b\in \mathbb N$ and $c\in \mathbb Z$, the physical interpretation of  the random variable $X_n^{a,b,c}$ is the total number of white balls obtained in $n$ extractions from an urn containing initially $a$ white balls and $b$ black balls, when the extractions are made with $c$ replacements (the extracted ball is returned to the urn together with $c$ balls of the same color, a negative value of $c$ being interpreted as removing $\vert c\vert$ balls from the urn).

It is known (e.g. \cite{Johnson-Kotz}) that the mean and variance of $X_n^{a,b,c}$  are given by
\begin{equation}
E\left( X_{n}^{a,b,c}\right) =\frac{na}{a+b}\qquad \text{and} \qquad \sigma
^{2}\left( X_{n}^{a,b,c}\right) =\frac{nab}{\left( a+b\right) ^{2}}\left( 1+\frac{%
\left( n-1\right) c}{a+b+c}\right) .  \label{Polya mean and variance}
\end{equation}

In \cite{PPT1} we considered the operator $R_n$ defined on the space of real-valued functions on $[0,1]$ by
\begin{eqnarray}  \label{Rational Bernstein operator}
R_{n}\left( f;x\right) &=&Ef\left( \frac{1}{n}X_{n}^{x,1-x,-\min \left\{
x,1-x\right\} /(n-1)}\right) \\
&=&\sum_{k=0}^{n}C_{n}^{k}\frac{x^{\left( k,-\min
\left\{ x,1-x\right\} /(n-1)\right) }\left( 1-x\right) ^{\left( n-k,-\min
\left\{ x,1-x\right\} /(n-1)\right) }}{1^{\left( n,-\min \left\{
x,1-x\right\} /(n-1)\right) }}f\left( \frac{k}{n}\right),\notag
\end{eqnarray}
and we showed that it satisfies pointwise estimates (in terms of the moduli of continuity of the function, of its first or second derivative) which improve the corresponding estimates for the classical Bernstein operator. We also showed that the operator $R_n$ satisfies a global Popoviciu-type inequality of the form (\ref{Popoviciu's error estimate for Bernstein}), with a constant slightly larger than the optimal constant found by Sikkema in the case of Bernstein operator.

In the next section we will show that the constant in the Popoviciu-type inequality for the operator $R_n$ is in fact strictly smaller than Sikkema's optimal constant for the Bernstein operator $B_n$. This suggests (although we do not have a proof) that the operator $R_n$ provides a better approximation than the Bernstein operator $B_n$, claim which is also supported by the numerical and theoretical results obtained in \cite{PPT1}.

\section{Main results}

The proof of our main result rests on the following inequality for the rising factorial (or equivalently, Pochammer symbol), which may be of independent interest.

\begin{lemma}
\label{Monotonicity of F^c}For any $x\in \left[ 0,1\right] $ and any
non-negative integers $n>1$ and $r\leq nx-\sqrt{n}$, we have%
\begin{equation}
\frac{x^{\left( r+1,c\right) }\left( 1-x\right) ^{\left( n-r,c\right) }}{%
1^{\left( n,c\right) }}\leq x^{r+1}\left( 1-x\right) ^{n-r}  \label{claim}
\end{equation}%
for any $c\in \left[ -\min \left\{ x,1-x\right\} /\left( n-1\right) ,0\right]
$.

Moreover, the above inequality is strict except for the case  $c=0$.
\end{lemma}

\begin{proof}
For $c=0$ and the claim becomes an identity, so we may assume $c\ne0$. The claim also holds true for $r=0$, since%
\[
\frac{x\left( 1-x\right) ^{\left( n,c\right) }}{1^{\left( n,c\right) }}%
=x\prod_{i=0}^{n-1}\frac{1-x+ic}{1+ic} < x\prod_{i=0}^{n-1}\left(
1-x\right) =x\left( 1-x\right) ^{n}.
\]

Since $r\leq nx-\sqrt{n}<n-1$, we may assume that $0<r<n-1$, which in
particular shows that all the factors appearing in the rising factorials on
the left of (\ref{claim}) are positive.

Taking logarithms and rearranging the terms, the claim (\ref{claim}) is
equivalent to%
\begin{equation}
\sum_{i=1}^{n-1}\ln \frac{1}{1+ic}\leq \sum_{i=1}^{r}\ln \frac{x}{x+ic}%
+\sum_{i=1}^{n-r-1}\ln \frac{1-x}{1-x+ic}.  \label{claim equivalent}
\end{equation}

It is easy to see that the function $\varphi :\left\{ \left(
u,t\right) \in \mathbb{R}^{2}:u,u+ct>0\right\} \rightarrow \mathbb{R}$
defined by%
\[
\varphi \left( u,t\right) =\ln \frac{u}{u+ct},
\]%
is convex and increasing in the variable $t$ (holding $u>0$ fixed), and it
is also satisfies $\varphi \left( \alpha u,\alpha t\right) =\varphi \left(
u,t\right) $ for any $\alpha >0$. Using this and an area comparison, we
obtain%
\begin{equation}
\sum_{i=1}^{r}\ln \frac{x}{x+ic}=\sum_{i=1}^{r}\varphi \left( x,i\right)
=r\sum_{i=1}^{r}\frac{1}{r}\varphi \left( \frac{x}{r},\frac{i}{r}\right)
> r\int_{0}^{1}\varphi \left( \frac{x}{r},t\right)
dt=\int_{0}^{1}r\varphi \left( 1,\frac{rt}{x}\right) dt,  \label{ineq1}
\end{equation}%
and similarly%
\begin{equation}
\sum_{i=1}^{n-r-1}\ln \frac{1-x}{1-x+ic}> \int_{0}^{1}\left( n-r-1\right)
\varphi \left( 1,\frac{\left( n-r-1\right) t}{1-x}\right) dt,  \label{ineq2}
\end{equation}%
and%
\begin{equation}
\sum_{i=1}^{n-1}\ln \frac{1}{1+ic}=n\sum_{i=1}^{n-1}\frac{1}{n}\varphi
\left( \frac{1}{n},\frac{i}{n}\right) < n\int_{1/n}^{1}\varphi \left(
\frac{1}{n},t\right) dt.  \label{ineq3}
\end{equation}

Using a substitution, we can simplify the last term above as follows%
\begin{eqnarray*}
&&n\int_{1/n}^{1}\varphi \left( \frac{1}{n},t\right) dt=n\left(
\int_{0}^{1}\varphi \left( \frac{1}{n},t\right) dt-\int_{0}^{1/n}\varphi
\left( \frac{1}{n},t\right) dt\right) \\
&=&n\left( \int_{0}^{1}\varphi \left( \frac{1}{n},t\right) dt-\frac{1}{n}%
\int_{0}^{1}\varphi \left( \frac{1}{n},\frac{t}{n}\right) dt\right)
=\int_{0}^{1}n\varphi \left( 1,nt\right) dt-\int_{0}^{1}\varphi \left(
1,t\right) dt.
\end{eqnarray*}

Combining the above with (\ref{ineq1}) -- (\ref{ineq3}), it follows that (%
\ref{claim equivalent}) holds if we prove the inequality%
\begin{equation}
\int_{0}^{1}\varphi \left( 1,nt\right) dt\leq \int_{0}^{1}\frac{1}{n}\varphi
\left( 1,t\right) dt+\frac{r}{n}\varphi \left( 1,\frac{rt}{x}\right) +\frac{%
n-r-1}{n}\varphi \left( 1,\frac{\left( n-r-1\right) t}{1-x}\right) dt.
\label{claim equivalent biss}
\end{equation}

The convexity of the function $\varphi $ in the second variable shows that
the right side of the above inequality is larger than%
\[
\int_{0}^{1}\varphi \left( 1,\frac{t}{n}\left( 1+\frac{r^{2}}{x}+\frac{%
\left( n-r-1\right) ^{2}}{1-x}\right) \right) dt,
\]
and using the monotonicity of the function $\varphi $ in the second variable
it follows that the inequality (\ref{claim equivalent biss}) holds provided
we show that%
\begin{equation}
nt\leq \frac{t}{n}\left( 1+\frac{r^{2}}{x}+\frac{\left( n-r-1\right) ^{2}}{%
1-x}\right)  \label{claim final}
\end{equation}%
for any positive integer $r<nx-\sqrt{n}$.

Extending the above inequality to positive real values $r\leq nx-\sqrt{n}$,
we have%
\[
\frac{d}{dr}\left( 1+\frac{r^{2}}{x}+\frac{\left( n-r-1\right) ^{2}}{1-x}%
\right) =\frac{2r}{x}-\frac{2(n-r-1)}{1-x} =\frac{2\left(
r-nx+x\right) }{x\left( 1-x\right) }<\frac{2\left( r-nx+\sqrt{n}\right) }{%
x\left( 1-x\right) }\leq 0,
\]%
and therefore the right hand side of (\ref{claim final}) is decreasing in $%
r\leq nx-\sqrt{n}$.

It follows that (\ref{claim final}) holds, provided we show%
\[
n^{2}\leq 1+\frac{\left( nx-\sqrt{n}\right) ^{2}}{x}+\frac{\left( n\left(
1-x\right) +\sqrt{n}-1\right) ^{2}}{1-x},
\]%
or equivalent%
\[
n^{2}\leq 1+n^{2}x-2n\sqrt{n}+\frac{n}{x}+n^{2}\left( 1-x\right) +2n\left(
\sqrt{n}-1\right) +\frac{n-2\sqrt{n}+1}{1-x}.
\]

Rearranging terms, we have equivalent%
\[
\frac{\left( 2n-1\right) x^{2}-2\left( n+\sqrt{n}-1\right) x+n}{x\left(
1-x\right) }\geq 0,
\]%
which holds true, since the numerator is a quadratic function of $x$ with a
negative dis\-cri\-mi\-nant $\Delta =-4\left( n-1\right) \left( \sqrt{n}-1\right)
^{2}<0$ for all $n>1$, thus concluding the proof.
\end{proof}

Using the above lemma, we can now prove the following.

\begin{theorem}
\label{thm for upper bound for optimal constant} There exists a constant $%
C\leq 1.08970< C_{opt}=1.0898873...$ such that for any conti\-nu\-ous function $f:\left[
0,1\right] \rightarrow \mathbb{R}$ and any $n>1$ we have%
\begin{equation}
\left\vert R_{n}\left( f;x\right) -f\left( x\right) \right\vert \leq C\omega
\left( n^{-1/2}\right) ,\qquad x\in \left[ 0,1\right] ,
\end{equation}%
where $\omega \left( \delta \right) =\omega ^{f}\left( \delta \right) $
denotes the modulus of continuity of $f$.
\end{theorem}

\begin{proof}
For $a\in \mathbb{R}$ denote by $]a[$ the largest integer strictly smaller
than $a$, that is $]a[=k\in \mathbb{Z}$ if $k<a\leq k+1$.

Using the definition of the modulus of continuity of $f$, it can be seen
that $\omega \left( \lambda \delta \right) \leq \left( 1+]\lambda \lbrack
\right) \omega \left( \delta \right) $ for any $\lambda \geq 0$ and $\delta
>0$ (see \cite{Sikkema}). Using this, with $\delta =n^{-1/2}$ and $\lambda
=\left\vert x-\frac{k}{n}\right\vert $, $k\in \left\{ 0,1,\ldots ,n\right\} $, and $c\in \left[ -\frac{\min \left\{ x,1-x\right\} }{n-1},0\right] $, we obtain%
\begin{eqnarray}
\left\vert Ef\left( \frac{1}{n}X_{n}^{x,1-x,c}\right) -f\left( x\right)
\right\vert &\leq& \left\vert \sum_{k=0}^{n}\left( f\left( \frac{k}{n}%
\right) -f\left( x\right) \right) p_{k,n}^{x,1-x,c}\right\vert  \label{aux 1}
\\
&\leq & \omega \left( n^{-1/2}\right) \left( 1+\sum_{k=0}^{n}\left] \frac{%
\left\vert x-\frac{k}{n}\right\vert }{n^{-1/2}}\right[ p_{n,k}^{x,1-x,c}%
\right)  \nonumber \\
&\leq &\omega \left( n^{-1/2}\right) \left( 1+\sqrt{n}\sum_{\substack{k\in \left\{
0,1,\ldots ,n\right\} :\\\left\vert x-\frac{k}{n}\right\vert
>n^{-1/2}}}\left\vert x-\frac{k}{n}\right\vert p_{n,k}^{x,1-x,c}\right) .
\nonumber
\end{eqnarray}%

For $x\in (\frac{1}{\sqrt{n}},1]$, denoting by $r=r\left( x\right) $ the
largest integer for which $x-\frac{r}{n}>n^{-1/2}$, or equivalent $r=r\left(
x\right) =]nx-\sqrt{n}[\in \left\{ 0,1,\ldots ,n-1\right\} $, and using
Lemma 3 in \cite{Kozniewska}, we have
\begin{eqnarray}
&&\sum_{\substack{k\in \left\{ 0,1,\ldots ,n\right\} :\\x-\frac{k}{n}>n^{-1/2}}}\left%
\vert x-\frac{k}{n}\right\vert p_{n,k}^{x,1-x,c}=\sum_{k=0}^{r}\left( x-%
\frac{k}{n}\right) p_{n,k}^{x,1-x,c}=\frac{1}{n}\sum_{k=0}^{r}\left(
nx-k\right) p_{n,k}^{x,1-x,c}  \quad\label{aux 2} \\
&&=\frac{1}{n}(r+1)p_{n,r+1}^{x,1-x,c}\left( 1-x+\left( n-r-1\right)
c\right) =C_{n-1}^{r}\frac{x^{\left( r+1,c\right) }\left( 1-x\right)
^{\left( n-r,c\right) }}{1^{\left( n,c\right) }}.  \nonumber
\end{eqnarray}

The above sum equal zero for $x\in \left[ 0,\frac{1}{\sqrt{n}}\right] $, and therefore we have%
\[
\sum_{\substack{k\in \left\{ 0,1,\ldots ,n\right\} :\\x-\frac{k}{n}>n^{-1/2}}}\left( x-%
\frac{k}{n}\right) p_{n,k}^{x,1-x,c}=F_{n}^{c}\left( x\right) ,\qquad x\in %
\left[ 0,1\right] ,
\]%
where $F_{n}^{c}:\left[ 0,1\right] \mathbb{\rightarrow R}$ is the function
defined by%
\begin{equation}
F_{n}^{c}\left( x\right) =\left\{
\begin{tabular}{ll}
$0,$ & $x\in \left[ 0,\frac{1}{\sqrt{n}}\right] $ \\
$C_{n-1}^{r}\frac{x^{\left( r+1,c\right) }\left( 1-x\right) ^{\left(
n-r,c\right) }}{1^{\left( n,c\right) }},$ & $x\in (\frac{1}{\sqrt{n}},1]$%
\end{tabular}%
\right. \text{, where }r=r\left( x\right) =]nx-\sqrt{n}[.  \label{definition of the function F^c}
\end{equation}%

For $x\in \lbrack 0,1-\frac{1}{\sqrt{n}})$, denoting $s=s\left( x\right) $
the smallest integer for which $x-\frac{s}{n}<-\frac{1}{\sqrt{n}}$, or
equivalent $s=\left[ nx+\sqrt{n}+1\right] \in \left\{ 1,\ldots ,n\right\} $,
and using again Lemma 3 in \cite{Kozniewska} and (\ref{Polya mean and
variance}), we have%
\begin{eqnarray}
&&\sum_{\substack{k\in \left\{ 0,1,\ldots ,n\right\} :\\x-\frac{k}{n}<-n^{-1/2}}}\left%
\vert x-\frac{k}{n}\right\vert p_{n,k}^{x,1-x,c}=\frac{1}{n}%
\sum_{k=s}^{n}\left( k-nx\right) p_{n,k}^{x,1-x,c}  =\frac{1}{n}\sum_{k=0}^{s-1}\left( nx-k\right)
p_{n,k}^{x,1-x,c}\quad \quad\\
&=& C_{n-1}^{s-1}\frac{x^{\left( s,c\right) }\left( 1-x\right)
^{\left( n-s+1,c\right) }}{1^{\left( n,c\right) }}.\notag
\end{eqnarray}

Denoting by $r^{\prime }=n-s$ and $x^{\prime }=1-x$, and using the property $%
-\left[ a+1\right] =]-a[$, we have%
\[
r^{\prime }=n-s=n-\left[ n\left( 1-x\right) ^{\prime }+\sqrt{n}+1\right] =-%
\left[ -nx^{\prime }+\sqrt{n}+1\right] =]nx^{\prime }-\sqrt{n}[,
\]%
and using the definition (\ref{definition of the function F^c}) of $%
F_{n}^{c} $, we can rewrite the sum above as follows%
\begin{equation}
\sum_{\substack{k\in \left\{ 0,1,\ldots ,n\right\} :\\x-\frac{k}{n}<-n^{-1/2}}}\left\vert
x-\frac{k}{n}\right\vert p_{n,k}^{x,1-x,c}=C_{n-1}^{r^{\prime }}\frac{\left(
x^{\prime }\right) ^{\left( r^{\prime }+1,c\right) }\left( 1-x^{\prime
}\right) ^{\left( n-r^{\prime },c\right) }}{1^{\left( n,c\right) }}%
=F_{n}^{c}\left( x^{\prime }\right) =F_{n}^{c}\left( 1-x\right) ,\label{aux 4}
\end{equation}%
for any $x\in \left[ 0,1\right]$  (note that for $x\in \left[ 1-\frac{1}{\sqrt{n}},1\right] $ the left
hand-side of the above equality also equals the right hand-side, both sides
being equal to zero).

Combining (\ref{aux 1}) -- (\ref{aux 4}) above, we obtain%
\[
\left\vert Ef\left( \frac{1}{n}X_{n}^{x,1-x,c}\right) -f\left( x\right)
\right\vert \leq \omega \left( n^{-1/2}\right) \left( 1+\sqrt{n}\left(
F_{n}^{c}\left( x\right) +F_{n}^{c}\left( 1-x\right) \right) \right) ,
\]%
for any $x\in \left[ 0,1\right]$ and  $c\in \left[ -\frac{\min \left\{ x,1-x\right\} }{n-1},0\right] $. Considering in particular $c=-\frac{\min \left\{ x,1-x\right\} }{n-1}$, and using Lemma \ref{Monotonicity of F^c} (which shows that $F_{n}^{c}\left( x\right)\leq F_{n}^{0}\left( x\right)$ for any $x\in[0,1]$), we obtain%
\begin{equation}\label{estimate using Sikkema's function}
\left\vert R_{n}\left( f;x\right) -f\left( x\right) \right\vert \leq \omega \left( n^{-1/2}\right) \left( 1+\sqrt{n}\left( F_{n}^{0}\left( x\right) +F_{n}^{0}\left( 1-x\right) \right) \right),\qquad x\in[0,1] .
\end{equation}

In \cite{Sikkema} (Section 4), the author obtained the following estimate %
\begin{equation}\label{Sikkema estimate}
 1+\sqrt{n}\left( F_{n}^{0}\left( x\right) +F_{n}^{0}\left( 1-x\right) \right) \leq 1.0897, \qquad x\in[0,1],
\end{equation}
valid any positive integer $n\neq6$. Combining this with (\ref{estimate using Sikkema's function}) proves the claim of the theorem for any positive integer $n\neq 6$.

To conclude the proof, we have left to consider the case $n=6$. First note that $c=-\frac{\min\{x,1-x\}}{5}=-\frac{x}{5}$ or $\frac{x-1}{5}$, depending whether $x\leq\frac12$ or $x>\frac12$, and the function $r=r\left( x\right) =]6x-\sqrt{6}[$ in the definition $F^c_6$ takes the value $k\in\{0,1,2,3\}$ for $x\in\left( \frac{1}{\sqrt{6}}+\frac{k}{6},\frac{1}{\sqrt{6}}+\frac{k+1}{6}\right]$, thus in order to estimate $F_6^c$ there are several cases to consider. For example, in the case $x\in\left( \frac{1}{\sqrt{6}},\frac12\right]$ we have $r=0$ and $c=-\frac{x}{5}$, and from (\ref{definition of the function F^c}) we obtain
\begin{equation*}
F_6^c(x)=\frac{x(1-x)^{(6,-x/5)}}{1^{(6,-x/5)}}<\frac{\frac12(1-\frac{1}{\sqrt{6}})^{(6,-\frac{1}{5\sqrt{6}})}}{1^{(6,-1/10)}}=\frac{193282 - 78887 \sqrt{6}}{6804}\approx 0.00721673,
\end{equation*}
and for $x\in\left( \frac12 , \frac{1}{\sqrt{6}}+\frac16 \right]$ we have $r=0$ and $c=\frac{x-1}{5}$, and we obtain
\begin{equation*}
F_6^c(x)=\frac{x(1-x)^{(6,\frac{x-1}{5})}}{1^{(6,\frac{x-1}{5})}}\equiv 0.
\end{equation*}

In the remaining three cases we have $c=\frac{1-x}{5}$, and proceeding similarly we obtain
\begin{equation*}
F_6^c(x)\leq\left\{
\begin{tabular}{ll}
$C_5^1 \frac{x^{\left(2,\frac{1-x}{5}\right)} (1-x)^{\left(5,\frac{1-x}{5}\right)}}{1^{\left(6,\frac{1-x}{5}\right)}}$, &$x\in \left(\frac{1}{\sqrt{6}}+\frac16,\frac{1}{\sqrt{6}}+\frac26\right ]$\\
$C_5^2 \frac{x^{\left(3,\frac{1-x}{5}\right)}(1-x)^{4,\frac{1-x}{4}}}{1^{6,\frac{1-x}{5}}}$, & $x\in\left( \frac{1}{\sqrt{6}}+\frac26, \frac{1}{\sqrt{6}}+\frac36\right]$ \\
$C_5^3 \frac{x^{\left(3,\frac{1-x}{5}\right)}(1-x)^{3,\frac{1-x}{5}}}{1^{6,\frac{1-x}{5}}}$, & $x\in\left( \frac{1}{\sqrt{6}}+\frac36, 1\right]$ \\
\end{tabular}
\right.
\leq 0.014271.
\end{equation*}

In all cases above we obtained $F_6^c(x)\leq 0.014271$ for $x\in[0,1]$, and therefore
\begin{equation*}
1+\sqrt{6}\left(F_6^c(x)+F_6^c(1-x)\right) \leq 1+2 \sqrt{6}\cdot 0.014271 \approx 1.0699134 < 1.0897, \quad x\in[0,1],
\end{equation*}
 concluding the proof of the theorem.

\end{proof}

The above result leaves open the problem of finding the value of the optimal constant $C$ in the above theorem. Although we do not have an answer here, we believe that the optimal constant is much smaller than the value hinted by the above theorem.

A second remark is that we believe that Lemma \ref{Monotonicity of F^c}, on which the above proof rests, can be improved to show that the left hand-side of the inequality (\ref{claim}) is in fact an increasing function of $c\geq -\frac{\min \left\{ x,1-x\right\} }{n-1}$. If this conjecture is correct, an argument similar to the one used in
the above proof would show that
\[
\frac{\left\vert Ef\left( \frac{1}{n}X_{n}^{x,1-x,c}\right) -f\left(
x\right) \right\vert }{\omega \left( n^{-1/2}\right) }
\]%
is in fact a monotone increasing function of $c\geq -\min \left\{ x,1-x\right\}/(n-1)$, hence among all P\'{o}lya-Bernstein type operators of the form
$$P_{n}^{x,1-x,c}(f;x) = Ef\left( \tfrac{1}{n}X_{n}^{x,1-x,c}\right),\qquad x\in[0,1],$$
the one that provides the best approximation in the class of continuous functions is the operator $R_{n}$ given by (\ref{Rational Bernstein operator}), which corresponds to $c=-\min \left\{ x,1-x\right\} /(n-1)$.

\section*{Acknowledgements}

The first author kindly acknowledges the support by a grant of the Romanian National Authority for Scientific Research, CNCS - UEFISCDI, project number PNII-ID-PCCE-2011-2-0015.

\end{document}